\documentclass[10pt,a4paper,twoside]{amsart}

\usepackage{amsmath,amssymb,url}

\numberwithin{equation}{section}

\theoremstyle{plain}
\newtheorem{theorem}{Theorem}[section]
\newtheorem{lemma}[theorem]{Lemma}

\newtheorem{corollary}[theorem]{Corollary}

\theoremstyle{definition}
\newtheorem{definition}[theorem]{Definition}

\newtheorem{example}[theorem]{Example}

\begin{document}


\title[Algebras  with Medial-like Functional Equations on Quasigroups]{Algebras  with Medial-like Functional Equations on Quasigroups}

\author[A. Ehsani]{Amir Ehsani}
\address{Department of Mathematics, Mahshahr Branch, Islamic Azad University, Mahshahr, Iran.}
\email{a.ehsani@mhriau.ac.ir}
\author[A. Krape\v{z}]{Aleksandar Krape\v{z}}
\address{Mathematical Institute of the Serbian Academy of Sciences and Arts, Knez Mihailova 36, 11001 Belgrade, Serbia.}
\email{sasa@mi.sanu.ac.rs}
\author[Yu. M. Movsisyan]{Yuri Movsisyan}
\address{Department of Mathematics and Mechanics,  Yerevan State University,  Alex Manoogian 1, Yerevan 0025,
Armenia.}
\email{yurimovsisyan@yahoo.com
 }



\maketitle

\begin{abstract}
We consider $14$  medial-like balanced  functional equations with four object variables for a pair
$(f, g)$ of binary quasigroup operations. Then, we prove that every algebra $(B; f, g)$ with quasigroup operations satisfying a medial-like balanced functional equation has a linear representation on an abelian group $(B; +)$.
\end{abstract}
\section{Introduction}

A binary algebra $\textbf{B}$ is an ordered pair $(B; F)$, where
$B$ is a nonempty set and $F$ is a family of binary operations
$f:B^2\rightarrow B$. The set $B$ is called the universe (base,
underlying set) of the algebra $\textbf{B}=(B; F)$. If $F$ is
finite, say $F=\{f_1,\ldots,f_k\}$, we often write $(B;
f_1,\ldots, f_k)$ for $(B; F)$, by \cite{28}. The algebra \textbf{B} is a
groupoid if it has only one binary operation.

A binary quasigroup is usually defined to be a groupoid
$(B; f)$ such that for any $a,b\in B$ there are unique solutions
$x$ and $y$ to the following equations:
 \[f(a,x)=b \ \ \ \text{and}\ \ \ f(y,a)=b, \]
by  \cite{pfl}.
If $(B; f)$ is quasigroup we say that $f$ is a quasigroup operation.
A loop is a quasigroup with unit $(e)$ such that
\[f(e,x)=f(x,e)=x.\] Groups are associative quasigroups, i.e. they
satisfy:
\[f(f(x,y),z)=f(x,f(y,z))\] and they necessarily contain a unit.
A quasigrouup is commutative if 
\begin{equation}\label{comm}
f(x,y)=f(y,x).
\end{equation}
Commutative groups are known as abelian groups.

A triple $(\alpha,\beta,\gamma)$ of bijections from a set $B$ onto
a set $C$ is called an isotopy of a groupoid $(B; f)$ onto a
groupoid $(C; g)$ provided \[\gamma f(x, y)=g(\alpha x,\beta
y)\] for all $x,y\in B$. 
$(C; g)$ is then called an isotope of
$(B; f)$, and groupoids $(B; f)$ and $(C; g)$ are
called isotopic to each other. An isotopy of $(B; f)$ onto
$(B; f)$ is called an autotopy of $(B; f)$.
Let $\alpha$ and $\beta$ be permutations of $B$ and let $\iota$ denote
the identity map on $B$. Then $(\alpha,\beta,\iota)$ is a
principal isotopy of a groupoid $(B; f)$ onto a groupoid
$(B; g)$ means that $(\alpha,\beta,\iota)$ is an isotopy of
$(B; f)$ onto  $(B; g)$.
Isotopy is a generalization of isomorphism. Isotopic image of a
quasigroup is again a quasigroup. A loop isotopic to a group is isomorphic to it.
Every quasigroup is isotopic to some loop i.e., it is a loop isotope.

A binary quasigroup $(B;f)$ is linear over an abelian group if
\[f(x,y)=\varphi x+a+\psi y,\] where $(B;+)$ is an abelian group, $\varphi$
and $\psi$ are automorphisms of  $(B;+)$ and $a\in B$ is
a fixed element.
  Quasigroup linear over an abelian group is also called a
 $T$-quasigroup.

Quasigroups are important algebraic (combinatorial, geometric) structures which
arise in various areas of mathematics and other disciplines. We mention just a few
of their applications: in combinatorics (as latin squares, see \cite{Denes}), in geometry (as nets/webs, see \cite{Belousov}), in statistics (see \cite{Fisher}), in special theory of relativity (see \cite{Ungar}), in coding theory and cryptography (\cite{Shcher}).

\section{Functional equations on quasigroups}

We use (object) variables $x, y, u, v, z$ and operation symbols (i.e. functional variables) $f, g$. We assume that all operation symbols represent quasigroup operations. 
A functional equation is an equality s = t, where s and t are terms with symbols of unknown operations occurring in at least one of them.

\begin{definition}
Functional equation $s=t$ is \textit{balanced} if every (object) variable appears exactly once in $s$ and once in $t$. 

\end{definition}

 \begin{example}
The following are various functional equations:
\begin{align}
&\label{ass}f(f(x,y),z)=f(x,f(y,z)), \\
&\label{med}f(f(x,y),f(u,v))=f(f(x,u),f(y,v)), \\
&\label{pse}f(f(x,y),f(u,v))=f(f(f(x,u),y),v), \\
&\label{dis}f(x,f(y,z))=f(f(x,y),f(x,z)), \\
&\label{tra}f(f(x,y),f(y,z))=f(x,z), \\
&\label{ide}f(x,x)=x. 
\end{align}
Associativity ($Eq. \eqref{ass}$), mediality ($Eq. \eqref{med}$) and pseudomediality ($Eq. \eqref{pse}$) are balanced,
transitivity ($Eq. \eqref{tra}$),  left distributivity ($Eq. \eqref{dis}$)
and idempotency ($Eq. \eqref{ide}$) are not.
 \end{example}


We can define $4!=24$ balanced functional equations with four object variables on a quasigroup $(B; f)$:
\begin{align}
\label{1-0}f(f(x,y),f(u,v))=f(f(x,y),f(u,v))\\
\label{1-00}f(f(x,y),f(u,v))=f(f(x,y),f(v,u))\\
\label{1-1}f(f(x,y),f(u,v))=f(f(x,u),f(y,v))\\
\label{1-2}f(f(x,y),f(u,v))=f(f(x,u),f(v,y))\\
\label{1-3}f(f(x,y),f(u,v))=f(f(x,v),f(y,u))\\
\label{1-4}f(f(x,y),f(u,v))=f(f(x,v),f(u,y))\\
\label{1-05}f(f(x,y),f(u,v))=f(f(y,x),f(u,v))\\
\label{1-06}f(f(x,y),f(u,v))=f(f(y,x),f(v,u))\\
\label{1-5}f(f(x,y),f(u,v))=f(f(y,u),f(x,v))\\
\label{1-6}f(f(x,y),f(u,v))=f(f(y,u),f(v,x))\\
\label{1-7}f(f(x,y),f(u,v))=f(f(y,v),f(x,u))\\
\label{1-8}f(f(x,y),f(u,v))=f(f(y,v),f(u,x))\\
\label{1-9}f(f(x,y),f(u,v))=f(f(u,x),f(y,v))\\
\label{1-10}f(f(x,y),f(u,v))=f(f(u,x),f(v,y))
\end{align}
\begin{align}
\label{1-11}f(f(x,y),f(u,v))=f(f(u,y),f(x,v))\\
\label{1-12}f(f(x,y),f(u,v))=f(f(u,y),f(v,x))\\
\label{1-013}f(f(x,y),f(u,v))=f(f(u,v),f(x,y))\\
\label{1-014}f(f(x,y),f(u,v))=f(f(u,v),f(y,x))\\
\label{1-13}f(f(x,y),f(u,v))=f(f(v,x),f(y,u))\\
\label{1-14}f(f(x,y),f(u,v))=f(f(v,x),f(u,y))\\
\label{1-15}f(f(x,y),f(u,v))=f(f(v,y),f(x,u))\\
\label{1-16}f(f(x,y),f(u,v))=f(f(v,y),f(u,x))\\
\label{1-015}f(f(x,y),f(u,v))=f(f(v,u),f(x,y))\\
\label{1-016}f(f(x,y),f(u,v))=f(f(v,u),f(y,x))
\end{align}

 The equation $\eqref{1-0}$ is trivial i.e., all quasigroups are solutions of this equation. The equations $\eqref{1-00}$, $\eqref{1-05}$, $\eqref{1-06}$, $\eqref{1-013}$, $\eqref{1-014}$ and $\eqref{1-015}$ are all equivalent to the equation $\eqref{comm}$; solutions are commutative quasigroups.

\begin{definition}
The functional equation $\eqref{1-1}$ is called \textit{medial identity} and every quasigroup satisfying medial identity is called \textit{medial quasigroup}.
\end{definition}

\begin{theorem}\label{quasimedial}
If $(B; f)$ is a medial quasigroup then there exists an abelian group $(B;+)$, such that
\[f(x, y)=\varphi(x)+c+\psi(y),\] where $\varphi,\psi\in Aut
(B; +)$, $\varphi\psi=\psi\varphi$ and $c\in Q$, by \cite{22}.
\end{theorem}

\begin{definition}
The functional equation $\eqref{1-16}$ is called \textit{paramedial identity} and every quasigroup satisfying medial identity is called \textit{paramedial quasigroup}.
\end{definition}

\begin{theorem}\label{quasipara}
 If $(B; f)$ is a paramedial quasigroup then there exists an abelian group $(B; +)$, such that
\[f(x, y)=\varphi(x)+c+\psi(y),\] where $\varphi,\psi\in Aut
(B; +)$, $\varphi\varphi=\psi\psi$ and $c\in Q$, by \cite{13}.
\end{theorem}

\begin{definition}
A balanced equation $s = t$ is \textit{Belousov} if for every subterm $p$ of $s$ ($t$) there is a subterm $q$ of $t$ ($s$) such that $p$ and $q$ have exactly the same variables.
\end{definition}

 \begin{example}
Functional equations $\eqref{comm}$, $\eqref{1-016}$ and the following are Belousov equations:
\begin{eqnarray*}
&&f(x,y)=f(x,y),\\
&&f(x,f(y,z))=f(f(z,y),x).
\end{eqnarray*}
The equations $\eqref{ass}-\eqref{pse}$ are non-Belousov.
 \end{example}

By \cite{krapez2005}, we have the following results:

\begin{theorem}
The solutions of the equation $\eqref{1-016}$ belong to the variety of $4$-palindromic quasigroups.
\end{theorem}

\begin{theorem}\label{thkrap2007}
The equations $\eqref{1-2}-\eqref{1-4}$, $\eqref{1-5}-\eqref{1-12}$ and $\eqref{1-13}-\eqref{1-15}$ are equivalent to commutative (para)mediality; solutions constitute the variety of commutative $T$-quasigroups (i.e., with  $\varphi=\psi$).
\end{theorem}

\section{Algebras with Medial-like functional equations}
As a generalization of the  functional equations $\eqref{1-0}-\eqref{1-016}$, let us consider the following balanced functional equations:
\begin{align}
\label{2-0}f(g(x,y),g(u,v))=g(f(x,y),f(u,v))\\
\label{2-00}f(g(x,y),g(u,v))=g(f(x,y),f(v,u))\\
\label{2-1}f(g(x,y),g(u,v))=g(f(x,u),f(y,v))\\
\label{2-2}f(g(x,y),g(u,v))=g(f(x,u),f(v,y))\\
\label{2-3}f(g(x,y),g(u,v))=g(f(x,v),f(y,u))\\
\label{2-4}f(g(x,y),g(u,v))=g(f(x,v),f(u,y))\\
\label{2-05}f(g(x,y),g(u,v))=g(f(y,x),f(u,v))\\
\label{2-06}f(g(x,y),g(u,v))=g(f(y,x),f(v,u))\\
\label{2-5}f(g(x,y),g(u,v))=g(f(y,u),f(x,v))\\
\label{2-6}f(g(x,y),g(u,v))=g(f(y,u),f(v,x))\\
\label{2-7}f(g(x,y),g(u,v))=g(f(y,v),f(x,u))\\
\label{2-8}f(g(x,y),g(u,v))=g(f(y,v),f(u,x))\\
\label{2-9}f(g(x,y),g(u,v))=g(f(u,x),f(y,v))\\
\label{2-10}f(g(x,y),g(u,v))=g(f(u,x),f(v,y))\\
\label{2-11}f(g(x,y),g(u,v))=g(f(u,y),f(x,v))\\
\label{2-12}f(g(x,y),g(u,v))=g(f(u,y),f(v,x))\\
\label{2-013}f(g(x,y),g(u,v))=g(f(u,v),f(x,y))\\
\label{2-014}f(g(x,y),g(u,v))=g(f(u,v),f(y,x))\\
\label{2-13}f(g(x,y),g(u,v))=g(f(v,x),f(y,u))\\
\label{2-14}f(g(x,y),g(u,v))=g(f(v,x),f(u,y))\\
\label{2-15}f(g(x,y),g(u,v))=g(f(v,y),f(x,u))\\
\label{2-16}f(g(x,y),g(u,v))=g(f(v,y),f(u,x))\\
\label{2-015}f(g(x,y),g(u,v))=g(f(v,u),f(x,y))\\
\label{2-016}f(g(x,y),g(u,v))=g(f(v,u),f(y,x))
\end{align}

\begin{definition}
A pair $(f, g)$ of binary operations is called:
\begin{enumerate}
    \item[$-$] 
\textit{medial pair of operations}, if the algebra $(B; f, g)$ satisfies the equation $\eqref{2-1}$.
    \item[$-$]
\textit{paramedial pair of operations}, if the algebra $(B; f, g)$ satisfies the equation $\eqref{2-16}$.
    \end{enumerate}
\end{definition}

\begin{definition}
A binary algebra $\textbf{B}=(B; F)$ is called:
\begin{enumerate}
    \item[$-$] 
\textit{medial algebra}, if every pair of operations of the algebra $\textbf{B}$ is medial (or, the algebra $\textbf{B}$ satisfying medial hyperidentity).
    \item[$-$] 
\textit{paramedial algebra}, if every pair of operations of the algebra $\textbf{B}$ is paramedial (or, the algebra $\textbf{B}$ satisfying paramedial hyperidentity).
    \end{enumerate}
\end{definition}

The following results are obtained in \cite{nazari} and \cite{ehsanimovsisyan} respectively.

\begin{theorem}\label{medial}
Let the set $B$, forms a quasigroup under the binary operations
$f$ and $g$. If the  pair of binary operations $(f, g)$ is
medial, then there exists a binary operation,$ +$, under which $B$
forms an abelian group and for arbitrary elements $x,y \in B$ we
have:
\[f(x,y)=\varphi_1(x)+\psi_1(y)+c_1,\]
\[g(x,y)=\varphi_2(x)+\psi_2(y)+c_2,\]
where $c_1, c_2$ are fixed elements of $B$, and $\varphi_i, \psi_i \in Aut(B;+)$ for
$i=1,2$, such that: $\varphi_1 \psi_2=\psi_2\varphi_1$, $ \varphi_2\psi_1=\psi_1\varphi_2$, $ \psi_1\psi_2=\psi_2\psi_1$ and $ \ \varphi_1\varphi_2=\varphi_2\varphi_1$.
The group $(B;+)$, is unique up to isomorphisms.
\end{theorem}

\begin{theorem}\label{param}
Let the set $B$, forms a quasigroup under the binary operations
$f$ and $g$. If the  pair of binary operations $(f, g)$  is
paramedial, then there exists a binary operation,$ +$, under which
$B$ forms an abelian group and for arbitrary elements $x,y \in B$
we have:
\[f(x,y)=\varphi_1(x)+\psi_1(y)+c_1,\]
\[g(x,y)=\varphi_2(x)+\psi_2(y)+c_2,\]
where $c_1, c_2$ are fixed elements of $B$, and $\varphi_i$ and
$\psi_i$ are automorphisms on the abelian group $(B; +)$ for
$i=1,2$, such that: $\varphi_1\varphi_2=\psi_2\psi_1$, $ \varphi_2\varphi_1=\psi_1\psi_2$, $ \varphi_1\psi_2=\varphi_2\psi_1$ and $\ \psi_1\varphi_2=\psi_2\varphi_1$.
The group $(B; +)$, is unique up to isomorphisms.
\end{theorem}

\begin{example}
Let $B=Z_2\times Z_2$. We denote the elements of the group $B$ as
follows: \[\{(0;0), (1;0), (0;1), (1;1)\}\] and the set of all automorphisms of $B$ as follow:
\[Aut B=\{\varepsilon, \varphi_2, \varphi_3, \varphi_4, \varphi_5, \varphi_6\},\] where
$\varepsilon=\left(%
\begin{array}{cc}
  1 & 0 \\
  0 & 1 \\
\end{array}%
\right)$, $\varphi_2=\left(%
\begin{array}{cc}
  1 & 0 \\
  1 & 1 \\
\end{array}%
\right)$, $\varphi_3=\left(%
\begin{array}{cc}
  1 & 1 \\
  0 & 1 \\
\end{array}%
\right)$, $\varphi_4=\left(%
\begin{array}{cc}
  0 & 1 \\
  1 & 0 \\
\end{array}%
\right)$, 
$ \varphi_5=\left(%
\begin{array}{cc}
  1 & 1 \\
  1 & 0 \\
\end{array}%
\right)$  and $\varphi_6=\left(%
\begin{array}{cc}
  0 & 1 \\
  1 & 1 \\
\end{array}%
\right)$. 

If $\varphi=\left(%
\begin{array}{cc}
  a & b \\
  c & d \\
\end{array}%
\right)$ and $X=\left(%
\begin{array}{cc}
  x_1 & x_2 \\
\end{array}%
\right)$ then \[\varphi(X)=\left(%
\begin{array}{cc}
  x_1 & x_2 \\
\end{array}%
\right) \left(%
\begin{array}{cc}
  a & b \\
  c & d \\
\end{array}%
\right).\]
Put \[f_1(x,y)=\varphi_2(x)+\varphi_3(y)+c,\] 
\[f_2(x,y)=\varphi_3(x)+\varphi_2(y)+c,\] 
\[f_3(x,y)=\varepsilon(x)+\varphi_5(y),\] 
\[f_4(x,y)=\varepsilon(x)+\varphi_6(y),\]   
where $c\in B$, then  the algebra $(B; f_1, f_2)$ is a paramedial algebra
with quasigroup operations which is not medial and $(B; f_3, f_4)$ is a medial
algebra with quasigroup operations which is not paramedial.
\end{example}

The balanced functional equations $\eqref{2-0}$, $\eqref{2-00}$, $\eqref{2-05}$, $\eqref{2-06}$, $\eqref{2-013}$, $\eqref{2-014}$,  $\eqref{2-015}$ and $\eqref{2-016}$ are Belousov, while the balanced functional equations $\eqref{2-1}-\eqref{2-4}$, $\eqref{2-5}-\eqref{2-12}$ and $\eqref{2-13}-\eqref{2-16}$ are non-Belousov.

\begin{definition}
If $(B;\cdot)$ is a group, then the bijection, $\alpha:
B\rightarrow B$, is called a \textit{holomorphism} of $(B;\cdot)$ if
\[
\alpha (x\cdot y^{-1}\cdot z)=\alpha x\cdot (\alpha y)^{-1}\cdot \alpha
z,
\]
 for every $x,y,z\in B$. 
\end{definition}

Note that this concept is equivalent to the concept of quasiautomorphism of groups, by \cite{26}. The set of all holomorphisms  of $(B;\cdot)$ is denoted by
$Hol(B;\cdot)$ and it is a group under the superposition of
mappings: $(\alpha \cdot\beta)x=\beta(\alpha x)$, for every $x\in
B$.

The following properties of holomorphisms were proved for Muofang loops in \cite{10}.

\begin{lemma}\label{hol1}
Let for bijections $\alpha_1, \alpha_2, \alpha_3$ on the group
 $(B;\cdot)$, the following identity is satisfied:
\[
\alpha_1(x\cdot y)=\alpha_2(x)\cdot\alpha_3(y).
\]
Then $\alpha_1, \alpha_2,
 \alpha_3\in Hol(B;\cdot)$.
\end{lemma}

\begin{lemma}\label{hol2}
Every holomorphism $\alpha$ of the group $(B;\cdot)$  has the
following form:
\[
\alpha x=\varphi x\cdot k
\]
where, $\varphi\in Aut(B;\cdot)$ and $k\in B$.
\end{lemma}

\begin{theorem}\label{main}
Let the set $B$ forms a quasigroup under the binary operations
$f$ and $g$. If the  pair $(f,g)$ of binary operations satisfies one of the balanced non-Belousov functional equations ($\eqref{2-2}-\eqref{2-4}$, $\eqref{2-5}-\eqref{2-12}$ or $\eqref{2-13}-\eqref{2-15}$), then there exists a binary operation '$+$' under which $B$
forms an abelian group and for arbitrary elements $x,y \in B$ we
have:
\[
f(x,y)=\varphi_1(x)+\psi_1(y)+c_1,\]
\[
g(x,y)=\varphi_2(x)+\psi_2(y)+c_2,
\]
where $c_1, c_2$ are fixed elements of $B$ and $\varphi_i, \psi_i \in Aut(B;+)$ for
$i=1,2$, such that:
\begin{description}
\item[$-$] $\varphi_1\varphi_2=\varphi_2\varphi_1$,  $\varphi_1\psi_2=\psi_2\psi_1$, $\psi_1\varphi_2=\varphi_2\psi_1$ and $\psi_1\psi_2=\psi_2\varphi_1$ for the equation $\eqref{2-2}$.
\item[$-$] $\varphi_1\varphi_2=\varphi_2\varphi_1$,  $\varphi_1\psi_2=\psi_2\varphi_1$, $\psi_1\varphi_2=\psi_2\psi_1$ and $\psi_1\psi_2=\varphi_2\psi_1$ for the equation $\eqref{2-3}$.
\item[$-$] $\varphi_1\varphi_2=\varphi_2\varphi_1$,  $\varphi_1\psi_2=\psi_2\psi_1$, $\psi_1\varphi_2=\psi_2\varphi_1$ and $\psi_1\psi_2=\varphi_2\psi_1$ for the equation $\eqref{2-4}$.
\item[$-$] $\varphi_1\varphi_2=\psi_2\varphi_1$,  $\varphi_1\psi_2=\varphi_2\varphi_1$, $\psi_1\varphi_2=\varphi_2\psi_1$ and $\psi_1\psi_2=\psi_2\psi_1$ for the equation $\eqref{2-5}$.
\item[$-$] $\varphi_1\varphi_2=\psi_2\psi_1$,  $\varphi_1\psi_2=\varphi_2\varphi_1$, $\psi_1\varphi_2=\varphi_2\psi_1$ and $\psi_1\psi_2=\psi_2\varphi_1$ for the equation $\eqref{2-6}$.
\item[$-$] $\varphi_1\varphi_2=\psi_2\varphi_1$,  $\varphi_1\psi_2=\varphi_2\varphi_1$, $\psi_1\varphi_2=\psi_2\psi_1$ and $\psi_1\psi_2=\varphi_2\psi_1$ for the equation $\eqref{2-7}$.
\item[$-$] $\varphi_1\varphi_2=\psi_2\psi_1$,  $\varphi_1\psi_2=\varphi_2\varphi_1$, $\psi_1\varphi_2=\psi_2\varphi_1$ and $\psi_1\psi_2=\varphi_2\psi_1$ for the equation $\eqref{2-8}$.
\item[$-$] $\varphi_1\varphi_2=\varphi_2\psi_1$,  $\varphi_1\psi_2=\psi_2\varphi_1$, $\psi_1\varphi_2=\varphi_2\varphi_1$ and $\psi_1\psi_2=\psi_2\psi_1$ for the equation $\eqref{2-9}$.
\item[$-$] $\varphi_1\varphi_2=\varphi_2\psi_1$,  $\varphi_1\psi_2=\psi_2\psi_1$, $\psi_1\varphi_2=\varphi_2\varphi_1$ and $\psi_1\psi_2=\psi_2\varphi_1$ for the equation $\eqref{2-10}$.
\item[$-$] $\varphi_1\varphi_2=\psi_2\varphi_1$,  $\varphi_1\psi_2=\varphi_2\psi_1$, $\psi_1\varphi_2=\varphi_2\varphi_1$ and $\psi_1\psi_2=\psi_2\psi_1$ for the equation $\eqref{2-11}$.
\item[$-$] $\varphi_1\varphi_2=\psi_2\psi_1$,  $\varphi_1\psi_2=\varphi_2\psi_1$, $\psi_1\varphi_2=\varphi_2\varphi_1$ and $\psi_1\psi_2=\psi_2\varphi_1$ for the equation $\eqref{2-12}$.
\item[$-$] $\varphi_1\varphi_2=\varphi_2\psi_1$,  $\varphi_1\psi_2=\psi_2\varphi_1$, $\psi_1\varphi_2=\psi_2\psi_1$ and $\psi_1\psi_2=\varphi_2\varphi_1$ for the equation $\eqref{2-13}$.
\item[$-$] $\varphi_1\varphi_2=\varphi_2\psi_1$,  $\varphi_1\psi_2=\psi_2\psi_1$, $\psi_1\varphi_2=\psi_2\varphi_1$ and $\psi_1\psi_2=\varphi_2\varphi_1$ for the equation $\eqref{2-14}$.
\item[$-$] $\varphi_1\varphi_2=\psi_2\varphi_1$,  $\varphi_1\psi_2=\varphi_2\psi_1$, $\psi_1\varphi_2=\psi_2\psi_1$ and $\psi_1\psi_2=\varphi_2\varphi_1$ for the equation $\eqref{2-15}$.
\end{description}
The group $(B; +)$ is unique up to isomorphisms.
\end{theorem}
\begin{proof}
Let $(B; f, g)$ be an algebra with the  pair $(f, g)$ of binary quasigroup operations satisfies the non-Belousov functional equation $\eqref{2-2}$. We use the main result of \cite{1}. Let
the set $B$ forms a quasigroup under six operations $A_i(x,y)$
(for $i=1,\dots,6$). If these operations satisfy the following equation:
\begin{equation}\label{t1}
A_1(A_2(x,y),A_3(u,v))=A_4(A_5(x,u),A_6(y,v)),
\end{equation}
for all elements $x$, $y$, $u$ and $v$ of the set $B$ then there exists an
operation '$+$' under which $B$ forms an abelian group isotopic to all these six quasigroups. And there exists eight
one-to-one mappings; $\alpha$, $\beta$, $\gamma$, $\delta$, $\epsilon$,
$\psi$, $\varphi$ and $\chi$ of $B$ onto itself such that:
\begin{eqnarray*}
&&A_1(x,y)=\delta x+\varphi y,\\
&&A_2(x,y)=\delta^{-1}(\alpha x+\beta y),\\
&&A_3(x,y)=\varphi^{-1}(\chi x+\gamma y),\\
&&A_4(x,y)=\psi x+\epsilon y,\\
&&A_5(x,y)=\psi^{-1}(\alpha x+\chi y),\\
&&A_6(x,y)=\epsilon^{-1}(\beta x+\gamma y).
\end{eqnarray*}
Now, let $A^*_i(x,y)=A_i(y,x)$ then, using this equality in
$\eqref{t1}$ we have:
\begin{equation}\label{t2}
A_1(A_2(x,y),A_3(u,v))=A_4(A_5(x,u),A^*_6(y,v)),
\end{equation} and so,
\[
A_6^*(x,y)=A_6(y,x)=\epsilon^{-1}(\beta y+\gamma
x)=
\epsilon^{-1}(\gamma x+\beta y),
\]
as $(B; +)$ is an abelian group.
So, let \[A_1=A_5=A_6^*=f,\]
  \[A_2=A_3=A_4=g.\] 
With these
assumptions, from the
equation $\eqref{t2}$ we can reach  to the equation $\eqref{2-2}$. Since $A_1=A_5$, we have:
\begin{eqnarray*}
&&\delta x+\varphi y=\psi^{-1}(\chi x+\alpha y)\\
&&\Rightarrow \psi(\delta x+\varphi y)=\chi x+\alpha y \\
&&\Rightarrow \psi( x+ y)=\chi (\delta^{-1} x)+\alpha
(\varphi^{-1}y)\\
&&\Rightarrow \psi\in Hol(B;+)
\end{eqnarray*}
by, Lemma $\ref{hol1}$.
Similarly, since $A_1=A_6^*$, we have: $\epsilon\in Hol(A; +)$.
Therefore, by Lemma $\ref{hol2}$ there exist $\varphi_2, \psi_2\in
Aut(B; +)$ such that:
\begin{eqnarray*}
 &&\psi
x=\varphi_2x+a,\\
&&\epsilon x=b+\psi_2x,
\end{eqnarray*}
where $a, b$ are fixed elements in $B$. Hence,
\begin{eqnarray*}
&&g(x,y)=A_4(x,y)=\psi x+\epsilon y=\\
&&\varphi_2x+a+b+\psi_2x=\varphi_2x+c_2+\psi_2x,
\end{eqnarray*}
where $c_2=a+b$ is a fixed element in $B$.
By the same manner, we can show that: $\delta,\varphi\in
Hol(B;+)$, since $A_2=A_4^*$ and $A_3=A_4^*$. So,  there exist $\varphi_1,\psi_1 \in Aut(B;+)$ such that:
\begin{eqnarray*}
&&\delta x=\varphi_1+d,\\
&&\varphi x=e+\psi_1,
\end{eqnarray*}
where $d$ and $e$ are fixed elements in $B$. Hence,
\[f(x,y)=A_1(x,y)=\delta x+\varphi y=\varphi_1x+c_1+\psi_1y,\]
where $c_1=d+e$ is a fixed element in $B$.
Now, put:
\begin{eqnarray*}
f(x,y)=\varphi_1(x)+\psi_1(y)+c_1,\\
g(x,y)=\varphi_2(x)+\psi_2(y)+c_2,
\end{eqnarray*}
in the equation $\eqref{2-2}$ then, it is easy to check that $\varphi_1\varphi_2=\varphi_2\varphi_1$,  $\varphi_1\psi_2=\psi_2\psi_1$, $\psi_1\varphi_2=\varphi_2\psi_1$ and $\psi_1\psi_2=\psi_2\varphi_1$.
The uniqueness of the group $(B;+)$ follows from the Albert
theorem: if  two groups are isotopic, then
they are isomorphic, by \cite{7}.

The proof is similar for the rest of  non-Belousov functional equations ($\eqref{2-3}$, $\eqref{2-4}$, $\eqref{2-5}-\eqref{2-12}$ and $\eqref{2-13}-\eqref{2-15}$).
\end{proof}

\begin{corollary}
Let $(B; F)$ be a binary  algebra with quasigroup
operations which satisfies one of the  non-Belousov functional equations ($\eqref{2-2}-\eqref{2-4}$, $\eqref{2-5}-\eqref{2-12}$ or $\eqref{2-13}-\eqref{2-15}$)  then there exists an abelian group $(B;+)$ such that
every operation $f_i\in F$ is represented by the following rule:
\[
f_i(x,y)=\varphi_i(x)+c_i+\varphi_i(y),
\]
where $c_i\in B$ and $\varphi_i\in Aut(B;+)$ such that $\varphi_i\varphi_j=\varphi_j\varphi_i$, for every $1 \leqslant i,j \leqslant\vert F\vert$.
The group $(B;+)$ is unique up to isomorphisms.
\end{corollary}
\begin{proof}
Let $(B; F)$ satisfies the non-Belousov functional equation $\eqref{2-2}$.
If $f_0\in F$ is a fixed operation then, by Theorem $\ref{thkrap2007}$, $f$ is
principally isotopic to the abelian group operation '$+$' on $B$.
Now, if $f_i\in F$ is any operation, then the pair of binary
operations $(f_0,f_i)$ satisfies the  non-Belousov functional equation $\eqref{2-2}$. Hence, $f_0$ and $f_i$ are
principally isotopic to another abelian group operation '$*$'
on $B$. Thus, by transitivity of isotopy, any operation $f_i$ is
principally isotopic to the same abelian group operation '$+$'.
Hence, according to the proof of previous theorem we have:
\[f_i(x,y)=\varphi_i(x)+\psi_i(y)+c_i,\]
where $c_i\in B$ and $\varphi_i,\psi_i\in Aut(B;+)$ such that $\varphi_i\varphi_j=\varphi_j\varphi_i$,  $\varphi_i\psi_j=\psi_j\psi_i$, $\psi_i\varphi_j=\varphi_j\psi_i$ and $\psi_i\psi_j=\psi_j\varphi_i$, for every $1 \leqslant i,j \leqslant\vert F\vert$. If $i=j$ then \[\psi_i^2=\psi_i\varphi_i \Rightarrow \psi_i=\varphi_i.\]
So, 
\[
f_i(x,y)=\varphi_i(x)+c_i+\varphi_i(y),
\]
for $1 \leqslant i \leqslant\vert F\vert$. By putting
\[
f_j(x,y)=\varphi_j(x)+c_j+\varphi_j(y),
\]\[
f_i(x,y)=\varphi_i(x)+c_i+\varphi_i(y),
\]
in $\eqref{2-2}$ we can obtain $\varphi_i\varphi_j=\varphi_j\varphi_i$, for every $1 \leqslant i,j \leqslant\vert F\vert$. 
The proof is similar for the rest of  non-Belousov functional equations $\eqref{2-3}$, $\eqref{2-4}$, $\eqref{2-5}-\eqref{2-12}$ and $\eqref{2-13}-\eqref{2-15}$.
\end{proof}

The Theorems $\ref{medial}-\ref{main}$ enable us to give a short proof of Theorems $\ref{quasimedial}$, $\ref{quasipara}$ and $\ref{thkrap2007}$:

\begin{corollary}
Every quasigroup $(B; f)$ satisfying one of the balanced non-Belousov functional equations $\eqref{1-1}-\eqref{1-4}$, $\eqref{1-5}-\eqref{1-12}$ or $\eqref{1-13}-\eqref{1-16}$ has a linear representation on an abelian group $(B; +)$:
\[f(x,y)=\varphi x+c+\psi y,\]
where $c\in B$ is fixed element and $\varphi$ and $\psi$ are automorphisms on the abelian group $(B; +)$ such that:
\begin{description}
\item[$-$] $\varphi\psi=\psi\varphi$ for the equation $\eqref{1-1}$.
\item[$-$] $\varphi=\psi$ for the equations $\eqref{1-2}-\eqref{1-4}$, $\eqref{1-5}-\eqref{1-12}$ and $\eqref{1-13}-\eqref{1-15}$.
\item[$-$] $\varphi^2=\psi^2$ for the equation $\eqref{1-16}$.
\end{description}
\end{corollary}
\begin{proof}
If $(B; f)$ satisfy the equation $\eqref{1-1}$ then by putting $f=g$ and using the Theorem $\ref{medial}$, the proof is obvious.
If $(B; f)$ satisfy one of the equations $\eqref{1-2}-\eqref{1-4}$, $\eqref{1-5}-\eqref{1-12}$ or $\eqref{1-13}-\eqref{1-15}$, then by  putting $f=g$ and using the Theorem $\ref{main}$, the proof is obvious. 
If $(B; f)$ satisfy the equation $\eqref{1-16}$ then by putting $f=g$ and using the Theorem $\ref{param}$, the proof is obvious.
\end{proof}


{The first author acknowledges the Mahshahr Branch,  Islamic Azad University for financial assistance
of the research project no. $12345$. The second author is supported by the Ministry
of Education, Science and Technological
Development of Serbia through projects
ON $174008$ and ON $174026$.}


\end{document}